\newif\ifarxiv
\arxivtrue  

\ifarxiv
    \documentclass[11pt,letterpaper]{article}
    \usepackage[margin=1in]{geometry}
    \usepackage[title]{appendix}

\else
    \documentclass[journal,12pt,onecolumn,draftclsnofoot]{IEEEtran}

\fi

\usepackage{setspace}

\usepackage[utf8]{inputenc} 
\usepackage[T1]{fontenc}
\usepackage{url}
\usepackage{ifthen}
\usepackage{mdframed}
\usepackage{cite}
\usepackage{graphicx}
\usepackage{subfig}
\usepackage{dsfont}
\usepackage{mathtools}
\usepackage{enumitem}
\usepackage{centernot}
\usepackage{bbm}
\usepackage{bm}
\usepackage{cite}
\usepackage{dsfont}
\usepackage{orcidlink}
\usepackage{amssymb, amsmath, amsthm, amsfonts}
\usepackage{xcolor}
\usepackage{xspace}
\usepackage{algpseudocode}
\usepackage{tabularx}
\usepackage{algorithm}
\usepackage{verbatim}
\usepackage[normalem]{ulem}

\newtheorem{thm}{Theorem}[section]

\newtheorem{prop}[thm]{Proposition}
\newtheorem{cor}[thm]{Corollary}

\theoremstyle{definition}
\newtheorem{defn}{Definition}[section]

\def\0{\mathbf{0}}

\def\bx{\mathbf{x}}

\def\bZ{\mathbf{Z}}
\def\1{\mathbf{1}}

\def\bX{\bm{X}}

\def\bS{\bm{S}}

\def\E{\mathbb{E}}
\def\P{\mathbb{P}}
\def\R{\mathbb{R}}

\def\cA{\mathcal{A}}
\def\cB{\mathcal{B}}

\def\cG{\mathcal{G}}
\def\cH{\mathcal{H}}
\def\cN{\mathcal{N}}

\def\cS{\mathcal{S}}

\def\weq{ \ = \ }
\def\wge{ \ \ge \ }
\def\wle{ \ \le \ }
\def\TV{d_{\operatorname{TV}}}
\def\kl{D_{\operatorname{KL}}}

\newcommand{\Erdos}{Erd\H{o}s}
\newcommand{\Renyi}{R{\'e}nyi\xspace}

\newcommand{\remove}[1]{}

\ifarxiv
\makeatletter
\def\thanks#1{\protected@xdef\@thanks{\@thanks
        \protect\footnotetext{#1}}}
\makeatother
\title{Detecting weighted hidden cliques}
\author{
    Urmisha Chatterjee\textsuperscript{1}  \and
    Karissa Huang\textsuperscript{2} \and
    Ritabrata Karmakar\textsuperscript{1} \and
    B. R. Vinay Kumar\textsuperscript{6}\and
    G\'{a}bor Lugosi\textsuperscript{3}
    \and
    Nandan Malhotra\textsuperscript{4} \and 
    Anirban Mandal\textsuperscript{5} \and
    Maruf Alam Tarafdar\textsuperscript{5}  
    \thanks{
\begin{minipage}{0.45\textwidth}
    \footnotesize
    \textsuperscript{1} Indian Statistical Institute, Kolkata, India. \newline
    \textsuperscript{2} University of California, Berkeley, USA.\newline
    \textsuperscript{3} 
    Department of Economics and Business, 
Pompeu  Fabra University, Barcelona, Spain;
ICREA, Pg. Lluís Companys 23, 08010 Barcelona, Spain;
Barcelona School of Economics.
 \newline
\end{minipage}
\begin{minipage}{0.5\textwidth}
    \footnotesize
    \textsuperscript{4} University of Leiden, The Netherlands. \newline
    \textsuperscript{5} Indian Statistical Institute, Delhi, India. \newline
    \textsuperscript{6} Indian Institute of Technology Bombay, Mumbai, India. \newline
\end{minipage}}
}
\date{}
\else
\title{Detecting weighted hidden cliques}

\author{
    Urmisha~Chatterjee,
    Karissa~Huang,
    Ritabrata~Karmakar,
    B.~R.~Vinay~Kumar,
    G\'{a}bor~Lugosi,
    Nandan~Malhotra,
    Anirban~Mandal,
    and~Maruf~Alam~Tarafdar%
    
    \thanks{U. Chatterjee and R. Karmakar are with the Indian Statistical Institute, Kolkata, India.}%
    \thanks{K. Huang is with the University of California, Berkeley, USA.}%
    \thanks{G. Lugosi is with ICREA, Pg. Lluís Companys 23, 08010 Barcelona, Spain; the Department of Economics and Business, Pompeu Fabra University, Barcelona, Spain; and the Barcelona School of Economics.}%
    \thanks{N. Malhotra is with the University of Leiden, The Netherlands.}%
    \thanks{A. Mandal and M. A. Tarafdar are with the Indian Statistical Institute, Delhi, India.}%
    \thanks{B. R. Vinay Kumar is with the Indian Institute of Technology Bombay, Mumbai, India.}%
}
\fi

\begin{document}
\maketitle

\abstract
We study a generalization of the classical hidden clique problem to graphs with real-valued edge weights. Formally, we define a hypothesis testing problem. Under the null hypothesis, edges of a complete graph on $n$ vertices are associated with independent and identically distributed edge weights from a distribution $P$. Under the alternative hypothesis, $k$ vertices are chosen at random and the edge weights between them are drawn from a distribution $Q$, while the remaining ones are sampled from $P$. The goal is to decide, upon observing the edge weights, which of the two hypotheses they were generated from. We investigate the problem under two different scenarios: (1) when $P$ and $Q$ are completely known, and (2) when there is only partial information of $P$ and $Q$. In the first scenario, we obtain statistical limits on $k$ when the two hypotheses are distinguishable, and when they are not. Additionally, in each of the scenarios, we provide bounds on the minimal risk of the hypothesis testing problem when $Q$ is not absolutely continuous with respect to $P$. We also provide computationally efficient spectral tests that can distinguish the two hypotheses as long as $k=\Omega(\sqrt{n})$ in both the scenarios.

\setstretch{1.3}

\section{Introduction}
The classical hidden clique detection problem is a hypothesis testing problem in which one needs to test whether a given graph $G$ is a realization of an \Erdos-\Renyi graph $\cG(n,p)$ or alternatively, a realization of $\cG(n,p)$ along with a planted clique on a subset of $k=k_n$ vertices that are chosen at random. The hidden clique (or the planted clique) problem is a canonical example of a problem that exhibits a statistical-computational gap in theoretical computer science. While it is possible to determine whether there is a clique of size $k_n=\Omega(\log n)$ using a brute-force search that runs in quasi-polynomial time, the best known polynomial time algorithms have only been shown to detect cliques of size $k_n = \Omega(\sqrt{n})$.

In this work, we study a generalization of the hidden clique problem to graphs with edge weights. Let $P$ and $Q$ be two distributions on the real line. The weighted hidden clique problem is a hypothesis testing problem that involves observing the complete
graph on $n$ vertices with real-valued edge weights. Under the null hypothesis, edges are associated with independent and identically distributed edge weights from a distribution $P$. Under the alternative hypothesis, $k$ vertices are chosen at random and the edge weights between them are drawn from a distribution $Q$, while the remaining are sampled from $P$. The goal is to decide, upon observing the edge weights, which of the two hypotheses they were generated from. We investigate the problem under two different scenarios: (1) when the distributions $P$ and $Q$ are completely known, and (2) when there is only partial information available of $P$ and $Q$. 

Our contributions are as follows. 
\begin{itemize}
    \item \textbf{Scenario 1}: When the distributions $P,Q$ are known, we show that, if $P\neq Q$, the two hypotheses are distinguishable, as long as $k_n$ grows at a logarithmic rate.
    If $Q$ is not absolutely continuous with respect to $P$, we show that the two hypotheses can always be distinguished as long as $k_n\to \infty$. On the other hand when $Q$ is absolutely continuous with respect to $P$, we obtain statistical limits on $k$ when the two hypotheses are distinguishable in terms of information divergence measures between $P$ and $Q$. We also provide computationally efficient spectral tests that can distinguish the two hypotheses as long as $k=\Omega(\sqrt{n})$.
    \item \textbf{Scenario 2}: When the distributions $P$ and $Q$ are unknown and $Q$ is not absolutely continuous with respect to $P$, we provide a polynomial-time test to distinguish the two hypotheses. Under the knowledge of only the means of the two distributions, we also provide a spectral test to distinguish the two hypotheses when  $k=\Omega(\sqrt{n})$.
\end{itemize}
Note that by taking $P$ to be the Bernoulli distribution with parameter $p$ and $Q$ to be the Dirac distribution at $1$, the problem reduces to the classical hidden clique problem described before.

The paper is organized as follows: 
In Section \ref{sec:relatedwork}, we briefly review the related literature.
Section \ref{sec:prob_setting} formally describes the problem setting and the main results under both the scenarios. Our main results are proved in Section \ref{sec:proofs}. Section \ref{sec:conc} concludes the paper and lists some directions for future work. 
In the Appendix we recall some basic facts results that are required to prove our main results.
Below, we provide the notation used in the rest of the article.

The set of vertices is denoted by $[n] := \{1,2,\cdots,n\}$. For any subset $S\subseteq[n]$, the set of edges connecting vertices in $S$ is denoted by $E(S) := \{(i,j): i,j \in S \text{ and }i<j\}$.

\section{Related work}
\label{sec:relatedwork}
The hidden clique problem has been investigated both in terms of detecting whether a planted clique is present as well as recovering it. In this section, we review some of the related work on the statistical and computational aspects of both problems. 


The hidden clique problem 
in \Erdos-\Renyi random graphs $G(n,p)$
(with $p$ constant)
dates back to Jerrum \cite{jerrum1992large}
and Ku\v{c}era \cite{kuvcera1995expected}, who already pointed out that, while hidden cliques of size $2\log_{1/p} n$ can be detected and recovered without computational restrictions, finding computationally efficient algorithms is a highly nontrivial challenge.
Alon, Krivelevich, and Sudakov \cite{alon1998finding} showed how spectral methods
can be used to find a hidden clique of size proportional to $n^{1/2}$.
Efficient non-spectral algorithms that work in the same regime were introduced by
Feige and Krauthgamer \cite{feige2000finding}, Feige and Ron \cite{FeRo10},
Ames and Vavasis \cite{ames2011nuclear},
and
Dekel, Gurel-Gurevich and Peres \cite{dekel2014finding}. The latter paper introduces a particularly
simple method that reconstructs
the hidden clique with  computational complexity of optimal order ($O(n^2)$).
Deshpande and Montanari \cite{DeMo15} show that cliques of size $\sqrt{n/e}(1+o(1))$
can be recovered, with high probability, by an algorithm of nearly optimal complexity.

The prevalence of the statistical-computational gap mentioned above is referred to as the \emph{planted clique conjecture} and is contrasted with other related conjectures in Hirahara and Shimizu \cite{hirahara2024planted}. 

Various attempts have been made to prove that it is impossible
to find hidden cliques of size $o(\sqrt{n})$ with computationally efficient methods.
Progress has been made in this direction by restricting the class of
allowed algorithms. For a sample of such results, see
Meka, Potechin, and Wigderson \cite{meka2015sum}, 
Montanari, Reichman, and Zeitouni \cite{MoReZe15}
Barak, Hopkins, Kelner, Kothari, Moitra, and Potechin \cite{barak2019nearly},
and
Feldman, Grigorescu, Reyzin, Vempala, and Xiao \cite{FeGrReVeXi17}.
Chen and Xu \cite{chen2016statistical} identify four different regimes of hardness of the planted clique problem.
Gamarnik and Zadik \cite{gamarnik2024landscape} and
Gamarnik \cite{gamarnik2021overlap}
study the so-called ``Overlap Gap
Property'' of the planted clique problem, providing further evidence
of the computational hardness of recovering planted cliques of size $o(\sqrt{n})$.

The hypothesis testing framework described in the introduction is also closely related to  anomaly detection in deterministic graphs such as trees Arias-Castro, Candès, Helgason, and Zeitouni  \cite{Arias2008Searching}, lattices Arias-Castro, Candès, and Durand \cite{arias2011detection},
Arias-Castro and Grimmett \cite{arias2013cluster}. Here, the edges are equipped with real-valued weights drawn from a normal distribution whose mean depends on whether the edge is between vertices that are among the chosen $k$ anomalous ones or not. 
In Addario-Berry, Broutin, Devroye, and Lugosi  \cite{addario2010combinatorial}, the setting is generalized. The planted clique problem can be viewed as an anomaly detection problem where the edges within the clique constitute the anomaly. The hypothesis testing problem on \Erdos-\Renyi graphs is discussed in \cite[Chapter 1]{lugosi2017lectures}.

The hidden clique problem can also be viewed as a special case of the planted dense subgraph problem where instead of a clique being planted under the alternate hypothesis, edges are generated with a probability $q>p$ between the $k_n$ chosen vertices. 
The planted dense subgraph problem has been extensively studied in \cite{arias2014community,butucea2013detection,HaWuXu2015a,verzelen2015community,hajek2017information,bresler2023detection,huleihel2022inferring,ElHu25}.

Recently, several variants and extensions of the hidden clique problem have been investigated. Structures such as the planted tree by Massoulié, Stephan, and Towsley \cite{massoulie2019planting}, planted Hamiltonian cycle by Bagaria, Ding, Tse, Wu, and Xu \cite{bagaria2020hidden}, planted dense cycle by Mao, Wein, and Zhang \cite{mao2024information}, planted matching by Moharrami, Moore, and Xu \cite{moharrami2021planted},
Addario-Berry et al. \cite{addario2026statistical}, planted $r$-colorable graphs by Louis, Paul, and Raghavendra \cite{louis2025robust}, planted stars by Narang, Perkins, and Wee \cite{narang2026optimal} and the planted bipartite graph by Rotenberg, Huleihel, and Shayevitz \cite{rotenberg2024planted} have been considered. Additionally, in Dhawan, Mao, and Wein \cite{dhawan2025detection}, the authors investigate the planted clique problem on a hypergraph, see also Alaluusua and  Kumar \cite{alaluusua2026planted} 
Algorithms for planted clique recovery have also been investigated on geometric graphs by Avrachenkov, Bobu,  Litvak, and Michielan \cite{avrachenkov2025planted}.
The problem considered here is also closely related to the \emph{weighted stochastic block model}, see 
Heimlicher, Lelarge and Massouli{\'e} \cite{heimlicher2012community},
Lelarge, Massouli{\'e}, and Xu \cite{lelarge2013reconstruction},
Xu, Jog, and Loh \cite{xu2020optimal}, where the emphasis is on community reconstruction, rather than detection.

The present work is an extension of the planted dense subgraph detection problem. Instead of having two parameters, $p$ and $q$, for edges within and outside the chosen set of $k$ vertices, we consider a weighted graph where the edge weights are sampled from two distributions $P$ and $Q$. The formal problem setting is described in the next section along with our main results.

\section{Problem setting and main results}
\label{sec:prob_setting}
Consider the complete graph on $n$ vertices with real-valued edge weights $\bX = (X_{ij})_{i,j\in[n]}$ generated according to one of two different distributions. Under the null hypothesis, the edge weights are independent and identically distributed (i.i.d.) from a probability measure $P$. The resulting weighted graph is denoted by $\cG(n,P)$. Under the alternative hypothesis, $k$ vertices are chosen uniformly at random and the edge weights between them follow the distribution $Q$, while the remaining edge weights follow the distribution $P$. All edge weights are independent and the model is denoted by $\cG(n,k,P,Q)$. 
We study the hypothesis testing problem
\begin{equation}
    \cH_0: \bX \sim \cG(n,P) \qquad \text{ vs. } \qquad \cH_1: \bX \sim \cG(n,k,P,Q)~.
    \label{eq:hypo_test}
\end{equation}
In the following, we use $\P_0$ and $\P_1$ for the distribution of $\bX$ under the null and alternative hypotheses, respectively.

A \emph{test} is a function $T\colon \R^{\binom{n}{2}} \to \{0,1\}$ that takes $\bX$ as input and outputs one of the two hypotheses. The \emph{risk} of test $T$ is defined as 
\begin{equation}
    R(T) \weq \P_0\big(T(\bX)=1\big)+\P_1\big(T(\bX)=0\big)~.
    \label{eq:risk_defn}
\end{equation}
As it is well known, the test that minimizes the risk for the binary hypothesis testing problem \eqref{eq:hypo_test} is the \emph{likelihood ratio test} defined by $T^*(\bX) = 1$ if and only if $L(\bX)>1$ where $L(\bx) := \frac{\P_1(\bx)}{\P_0(\bx)}$ is the likelihood ratio.

Our goal is to understand the behavior of the risk of the likelihood ratio test, $R(T^*)$, in terms of the model parameters, and provide computationally efficient tests, whenever possible. In particular, we wish to understand for what values of $k \equiv k_n$ does $R(T^*) \to 1 $ (i.e., the two hypotheses are asymptotically indistinguishable), and for what values of $k_n$ does $R(T^*) \to 0$ (i.e., there exists a test that solves the problem  \eqref{eq:hypo_test}). Furthermore we investigate the existence of polynomial-time algorithms to distinguish the two hypotheses.  We study the problem under two scenarios: when there is complete information about the distributions $P$ and $Q$, and when there is only partial information available. In the following two subsections, we describe the main results in each of these scenarios. All proofs may be found in Section \ref{sec:proofs}.

\subsection{Complete information of \texorpdfstring{$P$}{P} and \texorpdfstring{$Q$}{Q}}
In this subsection, we assume that the statistician has access to the distributions $P$ and $Q$. The difficulty of the hypothesis testing problem \eqref{eq:hypo_test} is governed by the degree of similarity between the two distributions $P$ and $Q$. If $P=Q$, then the two hypotheses are identical and one cannot differentiate between them. If $Q$ is not absolutely continuous with respect to $P$, then $R(T^*)$ is bounded away from $1$ for any $n$. This follows from the  following simple result. 
\begin{thm}
\label{thm:not_abs_cont}
    Suppose there exists $A\subset \R$ such that $Q(A)>0$ whereas $P(A)=0$. Then, for any $n \ge 2$, the optimal risk satisfies
    \[
    R(T^*) \wle (1-Q(A))^{\binom{k}{2}}~.
    \]
\end{thm}
\noindent 
In particular, as soon as $k \to \infty$, $R(T^*) \to 0$. Next, we consider the case when $Q$ is absolutely continuous with respect to $P$. The following proposition asserts that if the optimal risk approaches $0$, then necessarily $k$ must grow with $n$.
\begin{prop}
\label{prop:ktoinfty}
Let $P,Q$ be such that $Q$ is absolutely continuous with respect to $P$. If $R(T^*) \to 0$ as $n \to \infty$, then $k_n \to \infty$.
\end{prop}
\noindent Moreover, if $k_n\to \infty$, regardless of how slowly $k_n$ grows, there exist distributions for which the optimal risk vanishes with $n$.
\begin{prop}
    Given any sequence $k_n \to \infty$, there exist distributions $P$ and $Q$, absolutely continuous with respect to each other, such that $R(T^*) \to 0.$
\label{prop:PQconstruction}
\end{prop}
\noindent We provide an explicit construction of the distributions for which $R(T^*) \to 0$ in Section \ref{sec:proofofPQconstruction}. Further, when $Q$ is absolutely continuous with respect to $P$, we provide conditions on $k_n$ when the optimal risk \eqref{eq:risk_defn} 
\begin{enumerate}
    \item[(a)] converges to $0$ solving the hypothesis testing problem in \eqref{eq:hypo_test} and,
    \item[(b)] converges to $1$ rendering the two hypotheses indistinguishable.
\end{enumerate}
Our results are expressed in terms of standard divergences between probability measures $Q$ and $P$ such as the Kullback-Leibler divergence ($\kl(Q||P)$) and the $\chi^2$-divergence ($\chi^2(Q,P)$), 
see the Appendix for the definitions.
\begin{thm}
Let $P,Q$ be such that $Q$ is absolutely continuous with respect to $P$ and $\chi^2(Q||P) <\infty$. For any $\epsilon >0$,
\begin{itemize}
\item[(a)] if $k_n \geq \frac{(2 + \epsilon)\log n}{ \kl(Q||P)}$, then $ \lim_{n\to \infty} R(T^*) = 0$.
\item[(b)] if $k_n \leq \lfloor 2 \log_{\rho} n -2\log_{\rho} \log_{\rho} n-1+2\log_{\rho} e -\epsilon \rfloor$ where $\rho = \chi^2(Q|| P) + 1 $, then $\lim_{n\to \infty} R(T^*) = 1$.
\end{itemize}
\label{thm:abs_cont}
\end{thm}
In the classical hidden clique problem, the $\chi^2$-divergence and the KL~-divergence coincide to provide a sharp phase transition for detecting a clique.
It is an interesting problem to determine whether such a sharp phase transition exists in the general problem studied here, and if it does, what the exact location of the threshold is. The leading constants in our upper and lower bounds only match in some special cases.

It also follows from the proof of Theorem \ref{thm:abs_cont}, together with Theorem \ref{thm:not_abs_cont}, that the "critical" value of $k_n$ is at most logarithmic in $n$:

\begin{cor}
For all $P\neq Q$, there exists a constant $c>0$, such that if $k_n\ge c\log n$, then $R(T^*)\to 0.$
\end{cor}

The bound in Theorem \ref{thm:abs_cont}(a) is obtained by analyzing the \emph{scan} test, described in Section \ref{sec:scan_test} below. The scan test searches over all possible subsets of $k$ nodes, computes the likelihood ratio of the weights of edges within the subset, and decides whether it is from $P$ or $Q$. This is computationally intensive for large $n$. As an alternative, we propose a spectral algorithm on the weight matrix $\bX$ which runs in polynomial time and can solve \eqref{eq:hypo_test} as long as $k_n = \Omega(\sqrt{n})$. 

To define the spectral test, let $p(\cdot)$ and $q(\cdot)$ denote the densities of the distributions $P$ and $Q$ with respect to some common dominating measure. Then we may transform the entries of the matrix $\bX$ to obtain a matrix $\bZ = \left(Z_{ij}\right)_{n\times n}$, where $Z_{i,j} = \phi(X_{ij})$, with
\begin{equation}
    \phi(x) = \begin{cases}
    1 & \text{ if } x\in \cA \\
    0 & \text{ otherwise }
\end{cases},
\qquad \text{ where } 
\cA = \{x : p(x)>q(x)\}~.
\label{eq:transform}
\end{equation}
Note that the total variation distance between $P$ and $Q$ equals $\TV(P,Q) = |P(\cA)-Q(\cA)|$. The maximum eigenvalue of the centered matrix $\bZ-\E_0\bZ$ may be used as test statistic. We have the following performance bound:
\begin{thm}
    Let $\delta \in (0,1)$ and let $P,Q$ be two distributions with densities $p(\cdot), q(\cdot)$ with respect to a common dominating measure on $\R$. 
    Consider the test $T_1$ that accepts the null hypothesis if and only if $\|\bZ-\E_0\bZ\| \le 4\sqrt{(\log 9)n+\log (4/\delta)}$.
    Then the risk of the test satisfies $R(T_1) \le \delta$, whenever
    $$
    k>(1+o(1)) \frac{4\sqrt{(\log 9)n+\log (4/\delta)}}{\TV(P,Q)}~.
    $$
    \label{thm:Spectral_known_densities}
\end{thm}

\subsection{Partial information of \texorpdfstring{$P$}{P} and \texorpdfstring{$Q$}{Q}}
In this section, we investigate the case when there is only partial information available on the distributions $P$ and $Q$. Our first result addresses the case when $Q$ is not absolutely continuous with respect to $P$ but the distributions themselves are unknown. In this case, the test used in the proof of Theorem \ref{thm:not_abs_cont} is not feasible. Yet, as the next result shows, even in this scenario, there exists a test distinguishing the two hypotheses, as long as $k_n\to \infty$. Moreover, the test can be computed in polynomial time.

\begin{thm}
    Suppose that $P$ and $Q$ are unknown and  $Q$ is not absolutely continuous with respect to $P$. 
    Then, $R(T^*) \to 0$ if $k\to \infty$ and $k\le n/2$.
    \label{thm:unknown_P_Q}
\end{thm}

On the other hand, when $Q$ is  absolutely continuous with respect to $P$, we provide a spectral test for \eqref{eq:hypo_test} that uses only the means of $P$ and $Q$ and their common support. When the means are different and the distributions have bounded support, the maximum eigenvalue of the centered weight matrix $\|\bX-\E_0\bX\|$  exhibits a gap under the two hypotheses which can be exploited to design a test. 
\begin{thm}\label{thm:Spectral_bdd_diff_means}
    Let $\delta \in (0,1)$ and $P,Q$ be two distributions having a bounded support $[a,b]$ and different means $\mu_P \neq\mu_Q$. Consider the test $T_2$ that accepts the null hypothesis if and only if $\|\bX-\E_0\bX\| \le 4(b-a)\sqrt{(\log 9)n+\log (4/\delta)}$. Then the risk of the test satisfies $R(T_2) \le \delta$ whenever $k > (1+o(1)) \frac{4(b-a)\sqrt{(\log 9)n+\log (4/\delta)}}{|\mu_Q-\mu_P|}$. 
\end{thm}
When the means are the same, the transform in \eqref{eq:transform}
 separates the two means and the spectral test proposed in the statement of Theorem \ref{thm:Spectral_known_densities} can be employed. Although this seems to require complete information on densities $p(\cdot)$ and $q(\cdot)$, it suffices if a subset $\cB \subset \cA$ is provided where the density of $P$ is strictly greater than the density of $Q$. In such a case, the function $\phi$ is the indicator function of the set $\cB$ and the result of Theorem \ref{thm:Spectral_known_densities} holds with $\cA$ replaced by $\cB$.

\section{Proofs}
\label{sec:proofs}
In this section, we provide the proofs of all our results from Section \ref{sec:prob_setting}. Standard well-known results on the likelihood ratio and inequalities between information divergence measures used in the proofs are provided in Appendix \ref{sec:prelims}.

\subsection{Proof of Theorem \ref{thm:not_abs_cont}}
\begin{proof}

To derive the stated upper bound on the risk, consider the test such that $T = 1$ if and only if at least one of the edge weights belongs to $A$. The Type I error for this test is $0$ since $P$ puts no mass on $A$. The Type II error is the probability 
\[
\P_1\big(\cap_{e\in E}\{X_e\in A^c\}\big)=\frac{1}{\binom{n}{k}}\sum_{S:|S|=k}\prod_{e\in E(S)}Q(A^c)\prod_{e\in E\setminus E(S)}P(A^c)~.
\]
Since $P(A^c) = 1$, the risk of $T$ satisfies $R(T) = (1 - Q(A))^{\binom{k}{2}}$.
\end{proof}

\subsection{Proof of Proposition \ref{prop:ktoinfty}}
\begin{proof}
From Proposition \ref{prop:root_likelihood} we conclude that $R(T^{\ast}) \to 0$ if and only if $E_0[\sqrt{L(\bX)}] \to 0$. We also have, 
\begin{align*}
E_0[\sqrt{L(\bX)}] 
&= \int \sqrt{\frac{1}{\binom{n}{k}}\sum_{S:|S|=k}\prod_{e\in E(S)} \frac{q(x_e)}{p(x_e)}} \prod_{e \in E} p(x_e) \,d\bx \\
&= \int \sqrt{\frac{1}{\binom{n}{k}}\sum_{S:|S|=k}\prod_{e\in E(S)} q(x_e)p(x_e) \prod_{e\notin E(S)}p(x_e)^2} \,d\bx.    
\end{align*}
Using Jensen's inequality on the integrand,
$$E_0[\sqrt{L(\bX)}] \geq \int \frac{1}{\binom{n}{k}} \sum_{S:|S|=k} \sqrt{\prod_{e\in E(S)} q(x_e)p(x_e) \prod_{e\notin E(S)}p(x_e)^2}\,d\bx = \left(\int \sqrt{q(x)p(x)} \,dx \right)^{\binom{k}{2}}~.$$

Note that $\int \sqrt{p(x)q(x)} \,dx = BC(P,Q)\leq 1$. The equality holds only when the distributions are identical. Therefore, for two different distributions $P$ and $Q$ if $E[\sqrt{L(\bX)}] \to 0$ then $ k \to \infty.$ This proves the proposition.
\end{proof}

\subsection{Proof of Proposition \ref{prop:PQconstruction}}
\label{sec:proofofPQconstruction}
\begin{proof}
For any $n$, let $k'_n = \inf_{m \geq n} k_m$. Then $k'_n$ is non-decreasing and tends to $\infty$. Let $\tau_{\text{min}} = \min_{e \in E}X_e$. Take $P = \text{\emph{Unif}}(0,1)$, and define the density of $Q$ as follows:
$$ q(x) = \begin{cases}
    \left(1 - \frac{1}{k'_1}\right) + \frac{1}{2} & \text{if}\,\, \frac{1}{2} \leq x < 1, \\
    2^{m-1}\left(\frac{1}{k'_{m-1}} - \frac{1}{k'_{m}}\right) + \frac{1}{2} & \text{if}\,\, \frac{1}{2^m} \leq x < \frac{1}{2^{m-1}} \text{ and }m\ge 2, \\
    0 & \text{otherwise.}
\end{cases}
$$
It is easy to check that the previous expression is indeed a valid density supported on $[0,1]$, and that $P$ and $Q$ are absolutely continuous with respect to each other. Consider the  test $T_{\text{min}}$ that rejects $\cH_0$ if and only if $\tau_{\text{min}} < \frac{1}{2^n}$. The Type I error equals
$$
\P_0\left(\min_{e \in E} X_e < \frac{1}{2^n}\right) 
= 1 - \left(1-\frac{1}{2^n}\right)^{\binom{n}{2}} 
= 1-e^{-\binom{n}{2}\big[\frac{1}{2^n}+o\big(\frac{1}{2^n}\big)\big]} 
\to 0, \quad \text{ as } n\to \infty~.
$$
Next we consider the Type II error. For any $n$,
$$
\P_1\left(\min_{e \in E} X_e \geq \frac{1}{2^n}\right) 
= \left(1-\frac{1}{2^n}\right)^{\binom{n}{2} - \binom{k_n}{2}} \left(1-\int_0^{\frac{1}{2^n}} q(x)dx\right)^{\binom{k_n}{2}} 
\leq e^{-\binom{k_n}{2}\int_0^{\frac{1}{2^n}} q(x)dx}~.
$$
Since $k_n \geq k'_n$, 
$$\binom{k_n}{2}\int_0^{\frac{1}{2^n}} q(x)dx \geq \binom{k'_n}{2}\int_0^{\frac{1}{2^n}} q(x)dx = \frac{1}{2}\binom{k'_n}{2} \sum_{i=n}^\infty \left(\frac{1}{k'_i} - \frac{1}{k'_{i+1}}\right) = \frac{k'_n -1}{4} \to \infty~,$$ 
as $n \to \infty$. Thus,
$$
\lim_{n \to \infty}\P_1\left(\min_{e \in E} X_e \ge \frac{1}{2^n}\right) = 0~.$$
The risk of the test given above approaches $0$ as $n \to \infty$. Hence, $R(T_{\text{min}}) \to 0$ as $n \to \infty$.
\end{proof}

\subsection{Proof of Theorem \ref{thm:abs_cont}}
\label{sec:scan_test}
\begin{proof}
\textit{Part (a)}:
From Proposition \ref{prop:divergence_relations}(d), since $\chi^2(Q||P)<\infty$, $\kl(Q||P) < \infty$. Consider the test $T_{\text{scan}}$ that rejects $H_0$ if 
$$
T_{\text{scan}} := \max_{\stackrel{S \subseteq [n]}{|S|=k}}\prod_{e\in E(S)} \frac{q(X_e)}{p(X_e)} 
\ > \  n^k~.
$$
By the union bound and the independence of edge weights, as $k \to \infty$,
$$
\P_0\left(T_{\text{scan}} > n^k\right) 
\wle \sum_{\stackrel{S \subseteq [n]}{|S|=k}} \P_0\Bigg(\prod_{e\in E(S)}\frac{q(X_e)}{p(X_e)} > n^k\Bigg) 
\wle \frac{\binom{n}{k}}{n^k} 
\wle \frac{1}{k!} \to 0~,
$$
where the second inequality is due to Markov's inequality. 

Now consider the Type II error. If $T_{\text{scan}} \leq n^k$, then $\prod_{e\in E(S^*)} \frac{q(X_e)}{p(X_e)} 
\ \le \  n^k$ for the chosen subset $S^*$ containing $k$ vertices under $\cH_1$. Thus,
\begin{align*}
\P_1\bigg(T_{\text{scan}} \leq n^k\bigg) 
&\wle \P_1\Bigg(\prod_{e\in E(S^*)} \frac{q(X_e)}{p(X_e)} \leq n^k \Bigg) \\
 &\weq \P_1\Bigg(\frac{\binom{k}{2}}{\frac{k^2}{2}} \cdot\frac{1}{\binom{k}{2}}\sum_{e\in E(S^*)} \log \frac{q(X_e)}{p(X_e)} \leq \frac{2\log n}{k} \Bigg)~.    
\end{align*}
Conditionally on $S^*$, the random variables $(X_e)_{e\in E(S^*)} $ are i.i.d. with distribution $Q$. By the strong law of large numbers, conditional on $S^*$, $\frac{1}{\binom{k}{2}}\sum_{e\in E(S^*)} \log \frac{q(X_e)}{p(X_e)} \to \E_1 \Big[\log \frac{q(X_e)}{p(X_e)}\Big] = \kl(Q||P)$ almost surely as $k \to \infty$. Since $k\geq(2+\epsilon)\frac{\log n}{\kl(Q||P)} $ for some $\epsilon > 0$. Then $k\to \infty$ as $n\to \infty$, and $\frac{2\log n}{k} \leq \frac{\kl(Q||P)}{1+ \epsilon / 2}$. Therefore, using the strong law of large numbers and the bounded convergence theorem,
\begin{align*}
\P_1\bigg(T_{\text{scan}} \leq n^k\bigg) \wle 
\E_{S^*} \Bigg[ \P_1\Bigg(\frac{\binom{k}{2}}{\frac{k^2}{2}} \cdot\frac{1}{\binom{k}{2}}\sum_{e\in E(S^*)} \log \frac{q(X_e)}{p(X_e)} \leq  \frac{\kl(Q||P)}{1+ \epsilon / 2} \Bigg\vert S^*\Bigg) \Bigg] \to 0~,
\end{align*}
as $n \to \infty$. This proves Part (a) of the theorem. Note that when $\kl(Q||P)$ is infinite, the same proof is valid giving $R(T^*) \to 0$ as long as $k=\Omega(\log n)$.

\noindent \textit{Part (b)}:
Using the Cauchy-Schwarz inequality along with Proposition \ref{prop:tv_likelihood}, we obtain
$$ R(T^*) \ge 1-\sqrt{\E_0[(L(\bX)-1)^2]} = 1-\sqrt{\E_0[L(\bX)^2]-1}~,$$
since $\E_0[L(\bX)] = 1$. We now evaluate $\E_0[L(\bX)^2]$ as
\begin{align*}
    \E_0[L(\bX)^2] &= \frac{1}{\binom{n}{k}^2}\E_0\Bigg[\bigg(\sum_{|S|=k}\prod_{e\in E(S)} \frac{q(X_e)}{p(X_e)}\bigg)\bigg(\sum_{|T|=k}\prod_{e\in E(T)} \frac{q(X_e)}{p(X_e)}\bigg)\Bigg]\\
    &= \frac{1}{\binom{n}{k}^2}\E_0\Bigg[\sum_{|S|=|T|=k}\Bigg(\prod_{e\in E(S)\backslash E(T)} \frac{q(X_e)}{p(X_e)}\prod_{e\in E(T)\backslash E(S)} \frac{q(X_e)}{p(X_e)}\prod_{e\in E(S)\cap E(T)} \Big(\frac{q(X_e)}{p(X_e)}\Big)^2\Bigg)\Bigg]\\
    &=\frac{1}{\binom{n}{k}^2} \sum_{|S|=|T|=k} \prod_{e\in E(S)\backslash E(T)} \E_0\Bigg[\frac{q(X_e)}{p(X_e)}\Bigg]\prod_{e\in E(T)\backslash E(S)} \E_0\Bigg[\frac{q(X_e)}{p(X_e)}\Bigg]\prod_{e\in E(S)\cap E(T)} \E_0\Bigg[\Big(\frac{q(X_e)}{p(X_e)}\Big)^2\Bigg],
\end{align*}
where we have used the independence of edges. Note that $\E_0\Big[\frac{q(X)}{p(X)}\Big]=\int p(x)\frac{q(x)}{p(x)} \, dx=1$. Let $\E_0\Big[\frac{q(X)}{p(X)}^2\Big]:=\rho$. 
Then
\begin{align*}
    \E_0[L(\bX)^2] &= \frac{1}{\binom{n}{k}^2} \sum_{|S|=|T|=k} \rho^{|E(S)\cap E(T)|}\\
    &= \frac{1}{\binom{n}{k}^2} \sum_{i=0}^k \sum_{\stackrel{|S|=|T|=k}{|S\cap T|=i}} \rho^{\binom{i}{2}}\\
    &= \frac{1}{\binom{n}{k}^2} \sum_{i=0}^k \binom{n}{k} \binom{k}{i} \binom{n-k}{k-i} \rho^{\binom{i}{2}}
\end{align*}
Thus,
\begin{equation}
    \E_0[L(\bX)^2]-1 \le   \sum_{i=1}^k \frac{\binom{k}{i} \binom{n-k}{k-i} }{\binom{n}{k}} \rho^{\binom{i}{2}},
    \label{eq:lower_bd_1}
\end{equation}
which converges to $0$ whenever $k<\lfloor\omega_n-\epsilon\rfloor$ where $\omega_n= 2\log_\rho n-2\log_\rho \log_\rho n-1+2\log_\rho e$ (see \cite{Bol98}). This proves the theorem. 
Note that, $\rho$ can be represented using the chi-squared divergence between the two distributions as $\rho = \chi^2(Q,P)+1$ and also $\rho > 1$ by Jensen's inequality.
\end{proof}

\subsection{Proof of Theorem \ref{thm:Spectral_known_densities}}
\begin{proof}
    Computing the mean of an entry of $\bZ$ under the null hypothesis, we have that
    \[
    \E_0Z_{ij} = \E_0[\phi(X_{ij})] = \P_0(X_{ij}\in \cA)= P(\cA).
    \]
Similarly, under the alternate hypothesis $\E_1Z_{ij} = Q(\cA)< P(\cA)= \E_0Z_{ij}$. Thus, the mean of $Z_{ij}$ is different under both hypotheses. Moreover, $P(\cA)-Q(\cA) =\TV(P,Q)$ from Proposition \ref{prop:divergence_relations}. Since $Z_{ij}$ is distributed as a Bernoulli random variable with parameter $Q(\cA)$ under the alternative hypothesis, 
using Theorem \ref{thm:Spectral_bdd_diff_means} with $[a,b] = [0,1]$, we obtain the statement of the present theorem.
\end{proof}

\subsection{Proof of Theorem \ref{thm:unknown_P_Q}}
\begin{proof}
The key observation is that if $Q$ is not absolutely continuous with respect tp $P$, then there exists an interval that has $P$-measure zero, but positive $Q$-measure.
The proposed test searches for an interval $I$ for which there exists a set $S$ of $k$ vertices such that all edges whose weight falls in $I$ are connecting vertices in $S$. If there are sufficiently many such edges, then the test accepts the alternative hypothesis. More precisely,    
    let $T$ be the scan test over subsets of vertices $S$ with size $|S| = k$ as well as over intervals $I \subset \mathbb R$. In particular, define 
    \[
        \mathcal E = \left\{\exists S \subset [n] \text{ such that } |S| = k, \exists I\subset \R : \sum_{e \in E(S)} \bm 1_{\{X_e \in I\}} \geq k \text{ and } \sum_{e \not \in E(S)}\bm 1_{\{X_e \in I\}} = 0 \right\}~,
    \]
    and let $T=1$ if and only if $\mathcal E$ occurs. First, we bound the Type I error of $T$ as follows. First note that by the union bound and by symmetry, for an arbitrary fixed set $S$ of $k$ vertices,
    \begin{align*}
         \P_{0}(T = 1)
        &\leq \binom{n}{k} \P_{0}\left(\exists I :\sum_{e \in E(S)} \bm 1_{\{X_e \in I\}} \geq k \text{ and } \sum_{e \not \in E(S)}\bm 1_{\{X_e \in I\}} = 0\right)~.
    \end{align*}  
    Hence, it suffices to upper bound the probability on the right-hand side of the inequality above. 
     
We do this by an argument similar to the proof of an inequality of Vapnik and Chervonenkis \cite{VaCh74a} (see also \cite{BlEhHaWa89}). First note that since under the null hypothesis all edge weights are i.i.d.\, the probability is invariant to any permutation $(X_{\pi(e)})_{e\in E}$ of the weights. In particular, if $\Pi$ is a random permutation of $1,\ldots,\binom{n}{2}$, then
\begin{eqnarray*}
\lefteqn{
\P_{0}\left(\exists I :\sum_{e \in E(S)} \bm 1_{\{X_e \in I\}} \geq k \text{ and } \sum_{e \not \in E(S)}\bm 1_{\{X_e \in I\}} = 0\right) } \\
& = &
\E_0 \left[\P\left(\left. \exists I :\sum_{e \in E(S)} \bm 1_{\{X_{\Pi(e)} \in I\}} \geq k \text{ and } \sum_{e \not \in E(S)}\bm 1_{\{X_{\Pi(e)} \in I\}} = 0 \right| X_1,\ldots,X_{\binom{n}{2}} \right)\right]~. 
\end{eqnarray*}
In order to bound the conditional probability, observe that once the set   $\mathcal{X}=\{X_1,\ldots,X_{\binom{n}{2}}\}$ of the weights is fixed, there are at most 
$\binom{n}{2}(\binom{n}{2}+1)/2\le n^4$
different ways intervals can intersect this set. Let $\mathcal{I}$ denote a class of at most $n^4$ representative intervals (i.e., for any interval $I$ in $\R$ there is an interval  $I'\in \mathcal{I}$ such that $I\cap \mathcal{X}=I'\cap \mathcal{X}$).
By the union bound,
\begin{eqnarray*}
\lefteqn{
\P\left(\left. \exists I :\sum_{e \in E(S)} \bm 1_{\{X_{\Pi(e)} \in I\}} \geq k \text{ and } \sum_{e \not \in E(S)}\bm 1_{\{X_{\Pi(e)} \in I\}} = 0 \right| X_1,\ldots,X_{\binom{n}{2}} \right)  }\\
& \le & \sum_{I\in \mathcal{I}}
\P\left(\left. \sum_{e \in E(S)} \bm 1_{\{X_{\Pi(e)} \in I\}} \geq k \text{ and } \sum_{e \not \in E(S)}\bm 1_{\{X_{\Pi(e)} \in I\}} = 0 \right| X_1,\ldots,X_{\binom{n}{2}} \right)~.
\end{eqnarray*}
Consider now an interval $I\in \mathcal{I}$ with $\sum_{e\in E} \bm 1_{X_e \in I} =m$. If $m< k$
or $m>\binom{k}{2}$ then the probability above is clearly zero. Assume now that $k\le m \le \binom{k}{2}$. Then the proportion of permutations such that
$\sum_{e \in E(S)} \bm 1_{\{X_{\Pi(e)} \in I\}} \geq k$ and  $\sum_{e \not \in E(S)}\bm 1_{\{X_{\Pi(e)} \in I\}} = 0$ is at most
\[
  \frac{\left(\binom{n}{2}-m\right)! m!}{\binom{n}{2}!} \le
  \frac{1}{\binom{\binom{n}{2}}{k}}~. 
\]
Summarizing, the Type I error may be bounded by
\[
   \P_{0}(T = 1)
\leq n^4
\frac{\binom{n}{k}}{\binom{\binom{n}{2}}{k}}~,
\]
which goes to zero when $k\to \infty$.

    Next, we bound the Type II error. 
   Since $Q$ is not absolutely continuous with respect to $P$, there exists an interval $I\subset \R$ such that $P(I)=0$ and $Q(I)>0$.
   Since $k=k_n\to \infty$, for sufficiently large $n$, we have $Q(I)\ge 1/k$. Denoting the set of vertices of the planted clique by $S$, for all such $n$, we have
    \begin{align*}
        \P_{1}(T = 0) 
        &\leq \P_{1}\left(\sum_{e\in E(S)} \bm 1_{\{X_e \in I\}} < k\right) \\ 
        &\leq \P_{1}\left(\mathrm{Bin}\left(\binom{k}{2}, \frac{1}{k}\right) < k\right) \\
        &\leq \exp\left(-\frac{k-1}{4}\left(1 - \frac{2k}{k-1}\right)^2\right) \to 0~,
    \end{align*}
as desired.
\end{proof}

Observe that the test proposed in the proof can be computed in polynomial time, since there are $O(n^4)$ possible different intervals, and the defining property can be easily checked for each one of them.

\subsection{Proof of Theorem \ref{thm:Spectral_bdd_diff_means}}
\begin{proof}
    The spectral norm of the centered weight matrix is defined as 
 \[\|\bX-\E_0\bX\| \le \sup_{x\in \bS^{n-1}}\ |\langle x,(\bX-\E_0\bX) x\rangle|,\]
 where $\bS^{n-1}$ is the unit sphere in $\R^{n}$. Under the null hypothesis, construct an $\epsilon$-net $\cN$ with the property that for any $x\in \bS^{n-1}$,  there exists a $y\in \cN$ such that $\|x-y\|\le \frac{1}{4}$. From \cite[Theorem B.2]{lugosi2017lectures}, it can be shown that $|\cN| \le 9^n$. Additionally, the supremum in the definition of $\|\bX-\E_0\bX\|$ can be taken over the points in the countable set $\cN$ to produce
     \[\|\bX-\E_0\bX\| \le \sup_{x:x\in \cN}\ |\langle x,(\bX-\E_0\bX) x\rangle|.\]
     Alternately, let $x^*$ achieve the supremum in $\cS^{n-1}$. Let $C=\bX-\E_0\bX$. Then, we obtain for $y\in \cN$ such that $\|x^*-y\| \le \frac{1}{4}$,
     \begin{align*}
         \|C\| &= |\langle y,Cy\rangle +\langle x^*-y, Cx^*\rangle+\langle y,C(x^*-y)| \\
         &\le |\langle y,Cy\rangle| +2 \|C\| \|x^*-y\|\\
         &= |\langle y,Cy\rangle| +\frac{\|C\|}{2}.
     \end{align*}
where we use the Cauchy-Schwarz inequality. Since the weights are assumed to be bounded between $[a,b]$, $C_{ij}y_iy_j \in [-(b-a)|y_iy_j|,(b-a)|y_iy_j|]$.  Using Hoeffding's inequality, we obtain 
\begin{align*}
    \P_0(|\langle y,Cy \rangle| >t) &\weq \P_0\Big(\Big|\sum_{i<j} C_{ij}y_iy_j \Big|  > \frac{t}{2}\Big)\\
    &\wle 2\exp\bigg(\frac{-t^2}{8\sum_{i<j}(b-a)^2y_i^2y_j^2}\bigg)\\
    &\wle 2\exp\bigg(\frac{-t^2}{4(b-a)^2}\bigg),
\end{align*}
where we use 
\[
\sum_{i<j} y_i^2y_j^2 = \frac{1}{2}\Big[\sum_{i\neq j \in [n]} y_i^2y_j^2\Big] \wle \frac{1}{2}\Big[\sum_{i,j \in [n]} y_i^2y_j^2\Big] =\frac{1}{2}.
\]
Therefore 
\begin{equation}
    \P_0(||C|| > t) \wle 9^n \max_{y\in\cN}\P_0\Big(|\langle y,Cy\rangle|\ge \frac{t}{2}\Big) \wle 2 \cdot 9^n\cdot \exp\Big(\frac{-t^2}{16(b-a)^2}\Big) = \frac{\delta}{2}
    \label{eq:type1_spectral}
\end{equation}
when $t=4(b-a)\sqrt{(\log 9)n+\log (4/\delta)}$.

Under the alternative hypothesis, let $S$ be the hidden clique. We can reorder the rows and coloumns of $\bX$ so that the entries corresponding to the vertices in $S$ form the first $k\times k$ submatrix. Note that $\E_0\bX = \mu_P (J-I)$  where $J$ is a $n\times n$ matrix with all $1$s  and $I$ is the identity matrix. Let $x_S$ be a vector such that $(x_S)_{i} = \frac{1}{\sqrt{k}}$ if $i \in S$ and $(x_S)_{i} = 0$ otherwise. Then,
     \begin{align}
     \|\bX-\E_0\bX\| &\ge |\langle x_S,(\bX-\E_0\bX)x_S\rangle| \nonumber\\
     &= \bigg|\frac{1}{k} \sum_{i\neq j\in S} (X_{ij} - \mu_P) \bigg|\nonumber\\
     &= \bigg| \frac{2}{k} \binom{k}{2}(\mu_Q-\mu_P)+\frac{2}{k} \sum_{i< j\in S} (X_{ij} - \mu_Q) \bigg|\nonumber\\
     &\ge (k-1)\bigg[|\mu_Q-\mu_P|-\bigg|\frac{1}{\binom{k}{2}}\sum_{i< j\in S} (X_{ij} - \mu_Q)\bigg| \ \bigg].
     \label{eq:alt_norm_diff}
     \end{align}
     Conditional on $S$, the second term on the RHS in \eqref{eq:alt_norm_diff} is a sum of $\binom{k}{2}$ independent centered random variables. Since $Q$ has a bounded support, again using Hoeffding's inequality
     \begin{align*}
         \P_1\Bigg(\Big|\sum_{i< j\in S} (X_{ij} - \mu_Q)\Big| > \binom{k}{2}t\Bigg) &= \E_{S} \Bigg[ \P_1\Bigg(\Big|\sum_{i< j\in S} (X_{ij} - \mu_Q)\Big| > \binom{k}{2}t \Bigg\vert S \Bigg) \Bigg] \\
         &\le 2\exp\bigg[\frac{-t^2 \binom{k}{2}}{2(b-a)^2}\bigg]
     \end{align*}
      For any $\delta>0$, taking $t=(b-a)\sqrt{\frac{2}{\binom{k}{2}} \log \big(\frac{4}{\delta}\big)}$, and using \eqref{eq:alt_norm_diff}, we obtain
\begin{align}
    \P_1\Bigg(\|\bX-\E_0\bX\| \ge (k-1) \Bigg[|\mu_Q-\mu_P|&- (b-a)\sqrt{\frac{2}{\binom{k}{2}}\log \Big(\frac{4}{\delta}\Big)}\ \Bigg]\Bigg) \nonumber\\
    &\wge 
    \P_1\Bigg(\bigg|\frac{1}{\binom{k}{2}}\sum_{i< j\in S} (X_{ij} - \mu_Q)\bigg| \le (b-a)\sqrt{\frac{2}{\binom{k}{2}}\log \Big(\frac{4}{\delta}\Big)}\ \Bigg) \nonumber\\
    &\wge
    1-\frac{\delta}{2}.
    \label{eq:type2_spectral_bdd}
\end{align}
     Thus from \eqref{eq:type2_spectral_bdd} and \eqref{eq:type1_spectral}, we have that the risk $R(T_2) <\delta$ whenever 
     \[
     (k-1) \Bigg[|\mu_Q-\mu_P|- (b-a)\sqrt{\frac{2}{\binom{k}{2}}\log \Big(\frac{4}{\delta}\Big)}\ \Bigg] > 4(b-a)\sqrt{(\log 9)n+\log (4/\delta)}.
     \]
or when
$k > (1+o_n(1)) \frac{4(b-a)\sqrt{(\log 9)n+\log (4/\delta)}}{|\mu_Q-\mu_P|}
$.
\end{proof}

\section{Conclusions and future work}
\label{sec:conc}
In this article, we investigate the problem of detecting a planted subset of $k$ vertices sharing edge weights distributed according to $Q$ within a complete graph of $n$ vertices having edge weights distributed according to $P$. Under complete and partial information on the distributions $P$ and $Q$, we obtain statistical limits on $k$ when the clique can be detected and when it is impossible to detect it. 
We show that the critical value of $k$ for which detection becomes possible is at most logarithmic in $n$, regardless of the distributions. In some cases, much smaller planted cliques can be detected (e.g., when $Q$ is not absolutely continuous with respect to $P$). We also show that when $Q$ is absolutely continuous with respect to $P$, then the critical clique size $k$ converges to infinity, but any slow rate is possible, depending on $P$ and $Q$.
We also provide polynomial time spectral tests that can detect the clique. 

Our investigation of this problem leads to several interesting open questions.

\begin{enumerate}
    \item Theorem \ref{thm:abs_cont} provides bounds on $k$, when detection is possible and when it is impossible. However, there is a gap between the two bounds in parts (a) and (b) of the theorem, and it is unclear which of these, if any, is tight. More precisely, it is of interest to determine 
    whether there is always a sharp transition when detection becomes possible. If it is the case, then it is of interest to determine the correct "divergence" between the distributions $P$ and $Q$ that governs the weighted clique detection problem. While in the case of the classical hidden clique problem the two divergences ($\kl(Q||P)$ and $\chi^2(Q||P)$) coincide giving a sharp phase transition for clique detection, it is not obvious whether such a sharp transition exists for the weighted hidden clique problem.
    \item The upper bound for the risk in Theorem \ref{thm:abs_cont}(b) is obtained using a scan statistic that searches over all possible subsets of size $k$. 
    This procedure is computationally intensive but solves the detection problem \eqref{eq:hypo_test} for $k=\Omega(\log n)$. On the other hand, the spectral test proposed in Theorem \ref{thm:Spectral_known_densities} runs in polynomial time and works for $k=\Omega(\sqrt{n})$. Can this statistical-computational gap be closed for some distribution pairs $(P,Q)$, or can it be shown that such a non-trivial gap is inherent to the problem?
    \item In the present work, we only investigate the problem of detection.  The problem of recovering the subset of anomalous $k$ vertices is a natural extension. In particular, determining the statistical and computational thresholds for the recovery problem when there is partial or no information of the distributions is of interest. We leave these questions for future work.
\end{enumerate}

\section*{Acknowledgments}
The work was done in part during the workshop on the mathematical foundations of network models and their applications, as part of the BIRS-CMI pilot program organized by Louigi Addario-Berry, Siva Athreya, Shankar Bhamidi, Serte Donderwinkel and Soumik Pal. The workshop was held at the Chennai Mathematical Institute (CMI) from December 15-20, 2024 and the research school was conducted in the week preceding the BIRS event (December 9-13). The events were supported by BIRS, CMI, ICTS and NBHM.

UC acknowledges financial support from the Senior Research Fellowship Grant, Indian Statistical Institute, Kolkata.

KH was supported by the National Science Foundation
under grant DGE 2146752.

VK has received funding from the European Union's Horizon 2020 research and innovation programme under the Marie Skłodowska-Curie grant agreement, Grant Agreement No 101034253.

GL acknowledges the support of 
the Spanish Ministry of Economy and Competitiveness grant PID2022-138268NB-I00, financed by MCIN/AEI/10.13039/501100011033,
FSE+MTM2015-67304-P, and FEDER, EU.

NM is supported in part by the Netherlands Organisation for Scientific Research (NWO) through the Gravitation NETWORKS grant 024.002.003, and further supported by the European Union’s Horizon 2020 research and innovation programme under the Marie Skłodowska-Curie grant agreement no. 945045.

AM acknowledges financial support from the Senior Research Fellowship Grant, Indian Statistical Institute, Delhi.

MT acknowledges financial support from the Senior Research Fellowship Grant, Indian Statistical Institute, Delhi.

\appendix
\section{Appendix}
\label{sec:prelims}
In this section, we recall a few standard results that are used to prove the main results. 

\subsection{Expression for the likelihood ratio}
\label{sec:likelihood}
 Suppose $P$ and $Q$ admit densities $p(\cdot)$ and $q(\cdot)$ with respect to a common dominating measure on $\R$. Under the null hypothesis, the likelihood of observing the weights $\bx$ is given by 
\[
\frac{d\P_0}{d\bx} \weq \prod_{e\in E} p(x_{e})~.
\]
Under the alternative hypothesis, the likelihood of $\bx$ is
\[
\frac{d\P_1}{d\bx}  \weq \frac{1}{\binom{n}{k}}\sum_{\stackrel{S\subset [n]}{|S|=k}}  \prod_{e\in E(S)} q(x_{e}) \prod_{e\notin E(S)} p(x_{e})~.
\]
The likelihood ratio of the weighted graph $\bX$ is given by \begin{equation}
    L(\bX) = \frac{1}{\binom{n}{k}}\sum_{S:|S|=k}\prod_{e\in E(S)} \frac{q(X_e)}{p(X_e)}~.
    \label{eq:likelihood_ratio}
\end{equation}
Note that the existence of the densities with respect to the Lebesgue measure is not necessary. All our analysis holds even when the densities are with respect to some common dominating measure of $P$ and $Q$.

\subsection{Divergence measures between distributions}
We recall a few divergence measures between distributions $P$ and $Q$ which are used in the statements of our theorems. 

\begin{defn}
    For any two distributions $P, Q$ where $Q$ is absolutely continuous with respect to $P$, and having densities $p(\cdot), q(\cdot)$ with respect to a common dominating measure $\mu$ on $\R$, 
    \begin{itemize}
        \item \textbf{Total variation distance}:  $\TV(Q,P) = \sup_A |Q(A)-P(A)|.$
        \item \textbf{Kullback-Leibler (KL) divergence}: $\kl(Q||P) = \int q(x)\log \frac{q(x)}{p(x)} \, \mu(dx)$.
        \item \textbf{Chi-squared divergence}:  $\chi^2(Q || P) = \int \big(\frac{q(x)}{p(x)} - 1\big)^2 \, \mu(dx)$.
        \item \textbf{Squared Hellinger distance}: $H^2(Q, P) = \frac{1}{2}\int \left(\sqrt{q(x)}-\sqrt{p(x)}\right)^2 \, \mu(dx)$.
        \item \textbf{Bhattacharyya coefficient}: $ BC(Q, P) = \int\sqrt{q(x)p(x)} \, \mu(dx).$
    \end{itemize}
\end{defn}
Next, we list some useful properties and relationships between these divergence measures that we  invoke later. We omit the proof of this proposition since these are standard results (see, e.g. \cite{van_Erven_2014,janson2010asymptotic}.).
\begin{prop}
For any two distributions $P, Q$ where $Q$ is absolutely continuous with respect to $P$,
    \begin{itemize}
    \item[(a)] $\TV(P,Q) = \TV(Q,P)$, $H^2(P,Q) = H^2(Q,P)$.
    \item[(b)] $\kl(Q||P)  \wle \log \big(1+\chi^2(Q||P)\big) \wle \chi^2 (Q||P).$
    \item[(c)] $H^2(P,Q) \wle \TV(P,Q) \wle \sqrt{2}H(P,Q)$.
    \item[(d)] $H^2(P,Q) = 1-BC(P,Q)$.
    \end{itemize}
    Additionally, if $P, Q$ have densities $p(\cdot)$ and $q(\cdot)$ with respect to a measure $\mu$, then
    \begin{itemize} 
    \item[(e)] If $\cA = \{x:p(x)>q(x)\}$, then $\TV(P,Q) = P(\cA)-Q(\cA)$.
    \item[(f)] $\TV(P,Q) = \frac{1}{2} \int \big|p(x)-q(x)\big| \, \mu(dx)$.
    \end{itemize}
    \label{prop:divergence_relations}
\end{prop}

\subsection{Risk and likelihood ratio}
\label{sec:supp_lemmas}
Next, we recall some basic facts  relating the risk and the likelihood ratio.
\begin{prop}
    $R(T^*) = 1-\TV(\P_0,\P_1) = 1-\frac{1}{2}\E_0\big[|L(\bX)-1|\big] $
    \label{prop:tv_likelihood}
\end{prop}
\begin{proof}
See, e.g., \cite[Section 2.4]{Devroye1996}.
\end{proof}

\begin{prop}
    $1-\sqrt{1-\big[\E_0\sqrt{L(\bX)}\, \big]^2} \le R(T^*) \le \E_0[\sqrt{L(\bX)}]$.
    \label{prop:root_likelihood}
\end{prop}
\begin{proof}
From Propositions \ref{prop:tv_likelihood} and \ref{prop:divergence_relations}(c)-(d), $R(T^*) = 1-\TV(\P_0,\P_1) \le 1-H^2(\P_0,\P_1) = BC(\P_0,\P_1)$. Evaluating the Bhattacharyya coefficient
\begin{align*}
BC(\P_0,\P_1) &= \int \sqrt{\frac{d\P_0}{d\bx}  \frac{d\P_1}{d\bx} } \,d\bx\\
&=\E_0\big[\sqrt{L(\bX)}\,\big]~.
\end{align*}

The lower bound follows from the following inequality $\TV(\P_0, \P_1) \leq \sqrt{1 - BC(\P_0, \P_1)^2},$ which we now prove. For any $a, b \in \R$ it is straightforward to check that 
\[
    |a-b| = (a+b)\sqrt{1-\left(\frac{2\sqrt{ab}}{a+b}\right)^{2}}~.
\]
Applying this to total variation distance yields
\begin{align*}
    \TV(\P_0, \P_1) &= \frac{1}{2} \int \left|\frac{d\P_0}{d \bx} - \frac{d\P_1}{d\bx} \right| d\bx \\
    &= \frac{1}{2}\int \left(\frac{d\P_0}{d \bx} + \frac{d\P_1}{d\bx}\right) \sqrt{1 - \left(\frac{2\sqrt{\frac{d\P_0}{d\bx}\frac{d\P_1}{d\bx}}}{\frac{d\P_0}{d\bx} + \frac{d\P_1}{d\bx}}\right)^2} d\bx \\
    &= \frac{1}{2}\int \sqrt{1 - \left(\frac{2\sqrt{\frac{d\P_0}{d\bx}\frac{d\P_1}{d\bx}}}{\frac{d\P_0}{d\bx} + \frac{d\P_1}{d\bx}}\right)^2} d\mu~,
\end{align*}
where $d\mu = \frac{d\P_0 + d\P_1}{2}$. By concavity of the mapping $x\mapsto \sqrt{1-x^2}$ for $x\in[-1,1]$, Jensen's inequality yields
\[
    \TV(\P_0, \P_1) \leq \sqrt{1 - \int\left(\frac{2\sqrt{\frac{d\P_0}{d\bx}\frac{d\P_1}{d\bx}}}{\frac{d\P_0}{d\bx} + \frac{d\P_1}{d\bx}}\right)^2 d\mu} = \sqrt{1 - BC(\P_0, \P_1)^2}= \sqrt{1 - \left[\E_0\sqrt{L(\bX)}\right]^2}~.
\]
\end{proof}

\ifarxiv
\bibliographystyle{abbrv}
\else
\bibliographystyle{IEEEtran}
\fi
\bibliography{refs_hidden_clique}

@article{jerrum1992large,
  title={Large cliques elude the Metropolis process},
  author={Jerrum, Mark},
  journal={Random Structures \& Algorithms},
  volume={3},
  number={4},
  pages={347--359},
  year={1992},
  publisher={Wiley Online Library}
}

@article{feige2000finding,
  title={Finding and certifying a large hidden clique in a semirandom graph},
  author={Feige, Uriel and Krauthgamer, Robert},
  journal={Random Structures \& Algorithms},
  volume={16},
  number={2},
  pages={195--208},
  year={2000},
  publisher={Wiley Online Library}
}

@article{dekel2014finding,
  title={Finding hidden cliques in linear time with high probability},
  author={Dekel, Yael and Gurel-Gurevich, Ori and Peres, Yuval},
  journal={Combinatorics, Probability and Computing},
  volume={23},
  number={1},
  pages={29--49},
  year={2014},
  publisher={Cambridge University Press}
}

@article{arias2014community,
  title={Community detection in dense random networks},
  author={Arias-Castro, Ery and Verzelen, Nicolas},
  journal={The Annals of Statistics},
  pages={940--969},
  year={2014},
  publisher={JSTOR}
}

@inproceedings{meka2015sum,
  title={Sum-of-squares lower bounds for planted clique},
  author={Meka, Raghu and Potechin, Aaron and Wigderson, Avi},
  booktitle={Proceedings of the forty-seventh annual ACM symposium on Theory of computing},
  pages={87--96},
  year={2015}
}

@article{narang2026optimal,
  title={Optimal detection of planted stars via a random energy model},
  author={Narang, Ijay and Perkins, Will and Wee, Timothy LH},
  journal={arXiv preprint arXiv:2602.15585},
  year={2026}
}

@article{avrachenkov2025planted,
  title={Planted clique recovery in random geometric graphs},
  author={Avrachenkov, Konstantin and Bobu, Andrei and Litvak, Nelly and Michielan, Riccardo},
  journal={arXiv preprint arXiv:2510.12365},
  year={2025}
}

@article{butucea2013detection,
  title={Detection of a sparse submatrix of a high-dimensional noisy matrix},
  author={Butucea, Cristina and Ingster, Yuri I},
  journal={Bernoulli},
  pages={2652--2688},
  year={2013},
  publisher={JSTOR}
}

@article{hajek2017information,
  title={Information limits for recovering a hidden community},
  author={Hajek, Bruce and Wu, Yihong and Xu, Jiaming},
  journal={IEEE Transactions on Information Theory},
  volume={63},
  number={8},
  pages={4729--4745},
  year={2017},
  publisher={IEEE}
}

@article{verzelen2015community,
  title={Community detection in sparse random networks},
  author={Verzelen, Nicolas and Arias-Castro, Ery},
  year={2015}
}

@article{ames2011nuclear,
  title={Nuclear norm minimization for the planted clique and biclique problems},
  author={Ames, Brendan PW and Vavasis, Stephen A},
  journal={Mathematical programming},
  volume={129},
  number={1},
  pages={69--89},
  year={2011},
  publisher={Springer}
}

@article{barak2019nearly,
  title={A nearly tight sum-of-squares lower bound for the planted clique problem},
  author={Barak, Boaz and Hopkins, Samuel and Kelner, Jonathan and Kothari, Pravesh K and Moitra, Ankur and Potechin, Aaron},
  journal={SIAM Journal on Computing},
  volume={48},
  number={2},
  pages={687--735},
  year={2019},
  publisher={SIAM}
}

@article{alaluusua2026planted,
  title={Planted clique detection and recovery from the hypergraph adjacency matrix},
  author={Alaluusua, Kalle and Vinay Kumar, B R},
  journal={arXiv preprint arXiv:2604.08691},
  year={2026}
}

@article{janson2010asymptotic,
  title={Asymptotic equivalence and contiguity of some random graphs},
  author={Janson, Svante},
  journal={Random Structures \& Algorithms},
  volume={36},
  number={1},
  pages={26--45},
  year={2010},
  publisher={Wiley Online Library}
}

@inproceedings{bresler2023detection,
  title={Detection-recovery and detection-refutation gaps via reductions from planted clique},
  author={Bresler, Guy and Jiang, Tianze},
  booktitle={The Thirty Sixth Annual Conference on Learning Theory},
  pages={5850--5889},
  year={2023},
  organization={PMLR}
}

@article{kuvcera1995expected,
  title={Expected complexity of graph partitioning problems},
  author={Ku{\v{c}}era, Lud{\v{e}}k},
  journal={Discrete Applied Mathematics},
  volume={57},
  number={2-3},
  pages={193--212},
  year={1995},
  publisher={Elsevier}
}

@article{alon1998finding,
  title={Finding a large hidden clique in a random graph},
  author={Alon, Noga and Krivelevich, Michael and Sudakov, Benny},
  journal={Random Structures \& Algorithms},
  volume={13},
  number={3-4},
  pages={457--466},
  year={1998},
  publisher={Wiley Online Library}
}

@article{lugosi2017lectures,
  title={Lectures on combinatorial statistics},
  author={Lugosi, G{\'a}bor},
  journal={47th Probability Summer School, Saint-Flour},
  pages={1--91},
  year={2017}
}

@article{arias2011detection,
  title={Detection of an anomalous cluster in a network},
  author={Arias-Castro, Ery and Candes, Emmanuel J and Durand, Arnaud},
  journal={The Annals of Statistics},
  pages={278--304},
  year={2011},
  publisher={JSTOR}
}

@article{addario2010combinatorial,
 ISSN = {00905364, 21688966},
 URL = {http://www.jstor.org/stable/29765255},
 author = {Louigi Addario-Berry and Nicolas Broutin and Luc Devroye and Gábor Lugosi},
 journal = {The Annals of Statistics},
 number = {5},
 pages = {3063--3092},
 publisher = {Institute of Mathematical Statistics},
 title = {ON COMBINATORIAL TESTING PROBLEMS},
 urldate = {2025-01-29},
 volume = {38},
 year = {2010}
}

@article{van_Erven_2014,
   title={R\'{e}nyi Divergence and {Kullback-Leibler} Divergence},
   volume={60},
   ISSN={1557-9654},
   url={http://dx.doi.org/10.1109/TIT.2014.2320500},
   DOI={10.1109/tit.2014.2320500},
   number={7},
   journal={IEEE Transactions on Information Theory},
   publisher={Institute of Electrical and Electronics Engineers (IEEE)},
   author={van Erven, Tim and Harremoes, Peter},
   year={2014},
   month=jul, pages={3797–3820} }

@article{chen2016statistical,
  title={Statistical-computational tradeoffs in planted problems and submatrix localization with a growing number of clusters and submatrices},
  author={Chen, Yudong and Xu, Jiaming},
  journal={Journal of Machine Learning Research},
  volume={17},
  number={27},
  pages={1--57},
  year={2016}
}

@book {Devroye1996,
    AUTHOR = {Devroye, Luc and Gy\"orfi, L\'aszl\'o{} and Lugosi, G\'abor},
     TITLE = {A probabilistic theory of pattern recognition},
    SERIES = {Applications of Mathematics (New York)},
    VOLUME = {31},
 PUBLISHER = {Springer-Verlag, New York},
      YEAR = {1996},
     PAGES = {xvi+636},
      ISBN = {0-387-94618-7},
   MRCLASS = {68T10 (62G05 62H30 68T05)},
  MRNUMBER = {1383093},
MRREVIEWER = {Margaret\ P.\ Gessaman},
       DOI = {10.1007/978-1-4612-0711-5},
       URL = {https://doi.org/10.1007/978-1-4612-0711-5},
}

@article{Arias2008Searching,
author = {Ery Arias-Castro and Emmanuel J. Cand{\`e}s and Hannes Helgason and Ofer Zeitouni},
title = {{Searching for a trail of evidence in a maze}},
volume = {36},
journal = {The Annals of Statistics},
number = {4},
publisher = {Institute of Mathematical Statistics},
pages = {1726 -- 1757},
year = {2008},
doi = {10.1214/07-AOS526},
URL = {https://doi.org/10.1214/07-AOS526}
}

@article{arias2013cluster,
  title={Cluster detection in networks using percolation},
  author={Arias-Castro, Ery and Grimmett, Geoffrey R},
  journal={Bernoulli},
  pages={676--719},
  year={2013},
  volume = {19},
  number = {2},
  publisher={JSTOR}
}

@article{gamarnik2021overlap,
  title={The overlap gap property: A topological barrier to optimizing over random structures},
  author={Gamarnik, David},
  journal={Proceedings of the National Academy of Sciences},
  volume={118},
  number={41},
  pages={e2108492118},
  year={2021},
  publisher={National Academy of Sciences}
}

@article{gamarnik2024landscape,
  title={The landscape of the planted clique problem: Dense subgraphs and the overlap gap property},
  author={Gamarnik, David and Zadik, Ilias},
  journal={The Annals of Applied Probability},
  volume={34},
  number={4},
  pages={3375--3434},
  year={2024},
  publisher={Institute of Mathematical Statistics}
}

@article{dhawan2025detection,
  title={Detection of Dense Subhypergraphs by Low-Degree Polynomials},
  author={Dhawan, Abhishek and Mao, Cheng and Wein, Alexander S},
  journal={Random Structures \& Algorithms},
  volume={66},
  number={1},
  pages={e21279},
  year={2025},
  publisher={Wiley Online Library}
}

@article{huleihel2022inferring,
  title={Inferring hidden structures in random graphs},
  author={Huleihel, Wasim},
  journal={IEEE Transactions on Signal and Information Processing over Networks},
  volume={8},
  pages={855--867},
  year={2022},
  publisher={IEEE}
}

@inproceedings{massoulie2019planting,
  title={Planting trees in graphs, and finding them back},
  author={Massouli{\'e}, Laurent and Stephan, Ludovic and Towsley, Don},
  booktitle={Conference on Learning Theory},
  pages={2341--2371},
  year={2019},
  organization={PMLR}
}

@article{moharrami2021planted,
  title={The planted matching problem: Phase transitions and exact results},
  author={Moharrami, Mehrdad and Moore, Cristopher and Xu, Jiaming},
  journal={The Annals of Applied Probability},
  volume={31},
  number={6},
  pages={2663--2720},
  year={2021},
  publisher={Institute of Mathematical Statistics}
}

@article{bagaria2020hidden,
  title={Hidden Hamiltonian cycle recovery via linear programming},
  author={Bagaria, Vivek and Ding, Jian and Tse, David and Wu, Yihong and Xu, Jiaming},
  journal={Operations research},
  volume={68},
  number={1},
  pages={53--70},
  year={2020},
  publisher={INFORMS}
}

@article{mao2024information,
  title={Information-theoretic thresholds for planted dense cycles},
  author={Mao, Cheng and Wein, Alexander S and Zhang, Shenduo},
  journal={IEEE Transactions on Information Theory},
  volume={71},
  number={2},
  pages={1266--1282},
  year={2024},
  publisher={IEEE}
}

@article{rotenberg2024planted,
  title={Planted bipartite graph detection},
  author={Rotenberg, Asaf and Huleihel, Wasim and Shayevitz, Ofer},
  journal={IEEE Transactions on Information Theory},
  volume={70},
  number={6},
  pages={4319--4334},
  year={2024},
  publisher={IEEE}
}

@article{addario2026statistical,
  title={The statistical threshold for planted matchings and spanning trees},
  author={Addario-Berry, Louigi and Angel, Omer and Lugosi, G{\'a}bor and R{\'a}cz, Mikl{\'o}s Z and Schramm, Tselil},
  journal={arXiv preprint arXiv:2602.07669},
  year={2026}
}

@article{louis2025robust,
  title={Robust Algorithms for Recovering Planted r-Colorable Graphs},
  author={Louis, Anand and Paul, Rameesh and Raghavendra, Prasad},
  journal={Proceedings of Machine Learning Research vol},
  volume={291},
  pages={1--29},
  year={2025}
}

@inproceedings{hirahara2024planted,
  title={Planted clique conjectures are equivalent},
  author={Hirahara, Shuichi and Shimizu, Nobutaka},
  booktitle={Proceedings of the 56th Annual ACM Symposium on Theory of Computing},
  pages={358--366},
  year={2024}
}

@inproceedings{HaWuXu2015a,
	author = {Hajek, Bruce and Wu, Yihong and Xu, Jiaming},
	booktitle = {Conference on Learning Theory},
	organization = {PMLR},
	pages = {899--928},
	title = {Computational lower bounds for community detection on random graphs},
	year = {2015}}

@article{DeMo15,
  title={Finding hidden cliques of size  $\sqrt{N/e}$ in nearly linear time},
  author={Y. Deshpande and A. Montanari},
  journal={Foundations of Computational Mathematics},
  volume={15},
  number={4},
  pages={1069--1128},
  year={2015}
}

@inproceedings{FeRo10,
  title={Finding hidden cliques in linear time},
  author={U. Feige and D. Ron},
  booktitle={21st International Meeting on Probabilistic, Combinatorial, and Asymptotic Methods in the Analysis of Algorithms (AofA'10)},
  pages={189--204},
  year={2010},
  organization={Discrete Mathematics and Theoretical Computer Science}
}

@article{ElHu25,
  title={Detecting Arbitrary Planted Subgraphs in Random Graphs},
  author={Elimelech, Dor and Huleihel, Wasim},
  journal={arXiv preprint arXiv:2503.19069},
  year={2025}
}

@book{Bol98,
  title={Random graphs},
  author={Bollob{\'a}s, B{\'e}la },
  year={1998},
  publisher={Springer}
}

@BOOK{VaCh74a,
  title = { Theory of Pattern Recognition },
  publisher = { Nauka },
  year = { 1974 },
  author = { V.N. Vapnik and A.Ya. Chervonenkis },
  address = { Moscow },
  note = {(in Russian); German translation: {\em Theorie der Zeichenerkennung},
	Akademie Verlag, Berlin, 1979}
}

@article{BlEhHaWa89,
     author = {A. Blumer and  A. Ehrenfeucht and D. Haussler and
          M.K. Warmuth },
     title = {Learnability and the {V}apnik-{C}hervonenkis dimension},
     journal = { Journal of the ACM},
     year    = {1989},
     volume  = {36},
     pages = {929--965} }

@article{FeGrReVeXi17,
  title={Statistical algorithms and a lower bound for detecting planted cliques},
  author={V. Feldman and E. Grigorescu and L. Reyzin and S. Vempala and Y. Xiao},
  journal={Journal of the ACM (JACM)},
  volume={64},
  number={2},
  pages={8},
  year={2017},
  publisher={ACM}
}

@inproceedings{MoReZe15,
  title={On the limitation of spectral methods: From the gaussian hidden clique problem to rank-one perturbations of gaussian tensors},
  author={A. Montanari and D. Reichman and O. Zeitouni},
  booktitle={Advances in Neural Information Processing Systems},
  pages={217--225},
  year={2015}}

@article{xu2020optimal,
  title={Optimal rates for community estimation in the weighted stochastic block model},
  author={Xu, Min and Jog, Varun and Loh, Po-Ling},
  journal={The Annals of Statistics},
  volume={48},
  number={1},
  pages={183--204},
  year={2020}
}

@article{heimlicher2012community,
  title={Community detection in the labelled stochastic block model},
  author={Heimlicher, Simon and Lelarge, Marc and Massouli{\'e}, Laurent},
  journal={arXiv preprint arXiv:1209.2910},
  year={2012}
}

@inproceedings{lelarge2013reconstruction,
  title={Reconstruction in the labeled stochastic block model},
  author={Lelarge, Marc and Massouli{\'e}, Laurent and Xu, Jiaming},
  booktitle={2013 IEEE Information Theory Workshop (ITW)},
  pages={1--5},
  year={2013},
  organization={IEEE}
}

\end{document}